\documentclass[11pt]{article}
\usepackage{amsmath}
\usepackage{amssymb}
\usepackage{amsthm}
\usepackage{graphicx}
\usepackage{microtype}
\usepackage{hyperref}

\newtheorem{theorem}{Theorem}[section]
\newtheorem{lemma}[theorem]{Lemma}

\newtheorem{proposition}[theorem]{Proposition}
\newtheorem{corollary}[theorem]{Corollary}
\theoremstyle{definition}
\newtheorem{question}[theorem]{Question}

\newtheorem{definition}[theorem]{Definition}
\newtheorem{example}[theorem]{Example}
\newtheorem{thmx}{Theorem}

\newtheorem{remark}[theorem]{Remark}

\def\d{\displaystyle}

\def\g{\gamma}
\def\x{\xi}
\def\p{\partial}
\def\L{\mathcal{L}}
\def\P{\mathcal{P}}

\def\U{\mathcal{ULFP}}
\begin{document}
\title{\textbf{Uniform local finiteness of the curve graph via subsurface projections}}
\author{Yohsuke Watanabe}
\maketitle

\begin{abstract}
The curve graphs are not locally finite. In this paper, we show that the curve graphs satisfy a property which is equivalent to graphs being uniformly locally finite via Masur--Minsky's subsurface projections. 
As a direct application of this study, we show that there exist computable bounds for Bowditch's slices on tight geodesics, which depend only on the surface. 
As an extension of this application, we define a new class of geodesics, weak tight geodesics, and we also obtain a computable finiteness statement on the cardinalities of the slices on weak tight geodesics.

\end{abstract}
\vspace{0.1cm}



\section{\textbf{Introduction}}
Suppose $S=S_{g,n}$ is a compact surface of $g$ genus and $n$ boundary components. Let $\x(S_{g,n})=3g+n-3$ be the complexity of $S_{g,n}$. In this paper, we assume curves are simple, closed, essential, and non--peripheral. In \cite{HAR}, Harvey introduced the curve complex $C(S)$; suppose $\x(S)\geq 1.$ The vertices are isotopy classes of curves and the simplices are collections of curves that can be mutually realized to be disjoint. If $\x(S)=1$, any two distinct curves intersect at least once; the simplices are collections of curves that mutually intersect once if $S=S_{1,1}$ and twice if $S=S_{0,4}$. 

The main object in this paper will be the $1$--skeleton of $C(S)$, the curve graph; it is known to be path--connected and locally infinite. We put the simplicial metric (distance one for every edge) on the curve graph, then it is an infinite diameter graph with this metric \cite{MM1}. We denote the metric by $d_{S}$; for instance, let $x,y\in C(S)$, then $d_{S}(x,y)$ is the distance between $x$ and $y$, i.e., the length of a geodesic connecting $x$ and $y$.

The main result of this paper is to overcome the fact that the curve graph is locally infinite. This is a new observation in the studies in the curve complex, which is stated as Theorem \ref{br}. Roughly speaking, we show that the curve graph is indeed ``uniformly locally finite'' under the subsurface projections defined by Masur--Minksy \cite{MM2}. We refer the reader to $\S 2$ for the definition of subsurface projections. In order to motivate this study, we first observe Definition \ref{lfp}, Example \ref{firstsection} and Proposition \ref{LK} where we discuss in the setting of graphs.
Now, we define some terminologies on graphs ; let $X$ be a graph and $x \in X$. The valency at $x$ is the number of edges coming out of $x$, and we denote the valency at $x$ by $V(x)$. We say $X$ is locally finite if $V(x)<\infty$ for all $x\in X$, and $X$ is uniformly locally finite if there exists $P>0$ such that $V(x)\leq P$ for all $x\in X$. If $X$ is uniformly locally finite, then we let $V(X)=\max \{V(x)| x\in X\}$. In this paper, we assume a graph is path--connected and its diameter is infinite.

\begin{definition}[Uniform local finiteness property]\label{lfp}
Let $X$ be a graph with the simplical metric $d_{X}$. We say $X$ satisfies the uniform local finiteness property if there exists a computable $N_{X}(l,k)$ for any $l>0$ and $k>1$ such that the following holds. If $A\subseteq X$ such that $|A|>N_{X}(l,k)$, then there exists $A'\subseteq A$ such that $|A'|= k$ and $d_{X}(x,y)>l \text{ for all } x,y\in A'.$
\end{definition}

\begin{remark}
Note that $X$ in Definition \ref{lfp} is treated as the set of vertices rather than the graph; therefore, $A$ is a subset of the vertices of $X$. We usually treat graphs in this way throughout the paper; the interpretations will be clear from the context. Also, we often abbreviate the uniform local finiteness property as $\U$ in the rest of the paper. Lastly, we refer the reader to Proposition \ref{LK} for the reason why we call the above property the uniform local finiteness property.
\end{remark}

We consider the following examples to obtain the flavor of $\U$.

\begin{example}\label{firstsection}
Suppose $X$ is a graph. 
\begin{enumerate}
\item If $X$ is not locally finite, i.e., locally infinite, then $X$ does not satisfies $\U$. This is because we can take $x\in X$ such that $V(x)=\infty$, and consider the set of vertices which are distance $1$ apart from $x$; this set contains infinitely many element, but its diameter is bounded by $2.$

\item If $X$ is locally finite, but not uniformly locally finite, then $X$ does not satisfies $\U$. The reason is similar to the above.

\item Suppose $X$ is a tree such that $V(X)=2$. See Figure \ref{99}. We can think of the vertices of $X$ as $\mathbb{Z}$ in $\mathbb{R}$; if we have sufficiently many integers, we can find $k$ integers which are mutually more than $l$ apart in $\mathbb{R}$, i.e., $X$ satisfies $\U$. In particular, taking $N_{X}(l,k)=(l+2)k$ suffices.
\begin{figure}[htbp]
 \begin{center}
  \includegraphics[width=100mm]{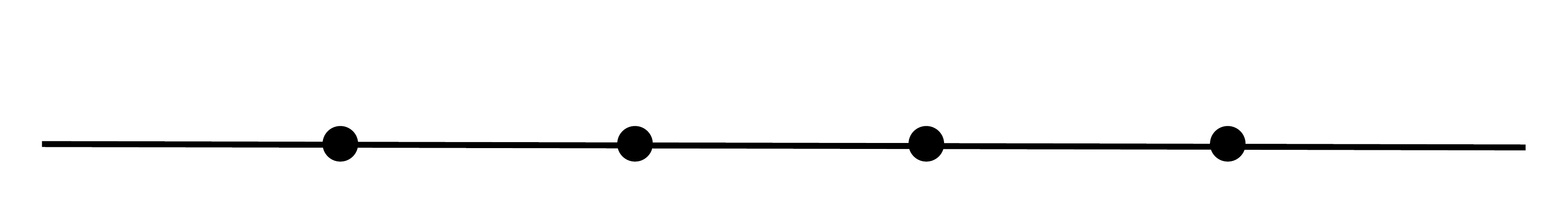}
 \end{center}
 \caption{A tree of valency $2$. }
 \label{99}
\end{figure}

\item If $X$ is uniformly locally finite, then $X$ satisfies $\U$. We prove it by showing the following equivalent statement.

Let $A\subseteq X$. If there is no $A'\subseteq A$ such that $|A'|= k$ and $d_{Z}(x,y)>l$ for all $x,y\in A'$, then there exists a computable $N_{X}(l,k)$ such that $|A|\leq N_{X}(l,k)$.
\begin{proof}
We prove the above statement by the induction on $k$; the proof works for all $l$.

\underline{Base case, $k=2$}: First, we define
\begin{itemize}
\item $N_{r}(x)=\{y\in X| d_{S}(x,y)\leq r\}$, the radius $r$ ball centered at $x$.
\item $C_{r}(x)=\{y\in X| d_{S}(x,y)=r\}$, the radius $r$ circle centered at $x$.
\end{itemize}
We observe that $\d |C_{i+1}(x)|\leq V(X)\cdot |C_{i}(x)|$ for all $i$. Therefore, we have $$ \d |N_{r}(x)|= \Bigg|\bigcup_{0\leq i\leq r} C_{i}(x)\Bigg| =  \sum_{i=0}^{r} |C_{i}(x)| \leq \sum_{i=0}^{r} V(X)^{i}.$$

Now, we prove the statement; since $A$ does not contain $2$ elements which are more than $l$ apart, we observe 
$$\d A\subseteq N_{l}(x) \text{ for some }x\in X. \Longrightarrow \d |A|\leq  \sum_{i=0}^{l} V(X)^{i}.$$

\underline{Inductive step, $k>2$}: Suppose $A$ does not contain $k$ elements which are mutually more than $l$ apart. Take $k-1$ elements $\{x_{i}\}_{i=1}^{k-1}\subseteq A$, which are mutually more than $l$ apart. (If there is no such elements, we are done by the inductive hypothesis.) Furthermore, we take $\{N_{l}(x_{i})\}_{i=1}^{k-1};$ since $A$ does not contain $k$ elements which are mutually more than $l$ apart, we observe  
$$\d A \subseteq \bigcup_{1\leq i \leq k-1} N_{l}(x_{i}).\Longrightarrow  |A|\leq(k-1) \cdot \sum_{i=0}^{l} V(X)^{i}.$$

\end{proof}

\end{enumerate}
\end{example}

We notice 
\begin{proposition}\label{LK}
$X$ is a uniformly locally finite graph. $\Longleftrightarrow$ $X$ satisfies $\U$.
\end{proposition}
\begin{proof} The statement follows from Example \ref{firstsection}.

$(\Longrightarrow)$: It follows by the forth example.

$(\Longleftarrow)$: It follows by the first and the second examples.
\end{proof}

The following is the main result of this paper. 

\begin{theorem}[$\U$ of the curve graph via subsurface projections]\label{br}
Suppose $\x(S)\geq 1$. There exists a computable $N_{S}(l,k)$ for any $l>0$ and $k>1$ such that the following holds. If $A\subseteq C(S)$ such that $|A|>N_{S}(l,k)$, then there exists $A'\subseteq A$ and $Z\subseteq S$ such that $|A'|= k$ and $d_{Z}(x,y)>l$ for all $x,y\in A'$.
\end{theorem}


As an application of a special case of Theorem \ref{br}, we study some finiteness statements on tight geodesics. Here, we recall the results regarding this study prior to this paper and some applications of them. We refer the reader to $\S 2$ for the definition of tight geodesics.

Tight geodesics were introduced by Masur--Minsky; they proved that there exists at least one and only finitely many tight geodesics between any pair of curves  \cite{MM2}. 

Bowditch defined a slice on all tight geodesics between a pair of (the sets of) curves, and showed that there exists a bound on the slice, which depends only on the surface. However, since his proof involves a geometric limit argument by using $3$--dimensional hyperbolic geometry, his bound was not computable \cite{BO2}. 

Schackleton showed that there exists a computable bound on the slice, which depends on the surface and the intersection number of a given pair of curves \cite{SHA}. 

Independently, Webb \cite{WEB2} and the author \cite{W1} showed that there exists a computable bound on the slice, which depends on the surface and the distance between a given pair of curves. By using this result, Webb showed that there exists a computable bound on the slice, which depends only on the surface. See Theorem \ref{webb's}. We note that his approach was combinatorial and constructive.

These studies of tight geodesics have many applications, including Thurston's ending lamination conjecture \cite{BCM} \cite{M1}, the asymptotic dimension of the curve graph \cite{BF}, and the stable lengths of pseudo--Anosov elements \cite{BO2} \cite{SHA} \cite{WEB2}.

Before we state our result, we recall an important geometric property of the curve graph. We first recall the following definition. 

\begin{definition}[Gromov]
We say a geodesic metric space is a $\delta$--hyperbolic space if every geodesic triangle in the space has the following property; any edge of a geodesic triangle is contained in the $\delta$--neighborhoods of other two edges. For instance, trees are $0$--hyperbolic spaces.
\end{definition}

It is known that the curve graph is $\delta$--hyperbolic; it was first proved by Masur--Minsky \cite{MM1}. By different approaches, Bowditch \cite{BO1} and Hamenst\"adt \cite{HAM} also proved this result. Recently, the hyperbolicity has been improved so that it is uniform for all surfaces. This result was independently proved by Aougab \cite{AOU}, by Bowditch \cite{BO4}, by Clay--Rafi--Schleimer \cite{CRS}, and by Hensel--Przytycki--Webb \cite{HPW}. In the rest of this paper, we let $\delta$ denote the hyperbolicity constant of the curve graph. 

Now, we review the definition of a slice from \cite{BO2}. Suppose $x\in C(S)$; we let $N_{i}(x)=\{y\in C(S)|d_{S}(x,y)\leq i\}.$
\begin{definition}[\cite{BO2}]
Suppose $a,b\in C(S)$, $A,B\subseteq C(S)$, and $r>0$. 
\begin{itemize}
\item Let $\L_{T}(a,b)$ be the set of all tight geodesics between $a$ and $b$.

\item Let $\d G(a,b)=\{v\in C(S)| v\in g\in \L_{T}(a,b)\}$.

\item Let $\d G(A,B)=\bigcup_{a\in A, b\in B} G(a,b)$ and $G(a,b;r)=G(N_{r}(a), N_{r}(b))$. 
\end{itemize}
\end{definition}

The following result is due to Bowditch \cite{BO2} without computable bounds. Here, we state the recent result by Webb; he also showed that his bounds are sharp. Suppose $a,b\in C(S)$; we let $g_{a,b}$ denote a geodesic between $a$ and $b$.

\begin{theorem}[\cite{WEB2}]\label{webb's}
Suppose $\x(S)\geq 2$. Let $a,b\in C(S)$, $r\geq 0$, and $K$ be a uniform constant.
\begin{enumerate}
\item  If $c\in g_{a,b}$, then $|G(a,b)\cap N_{\delta}(c)|\leq K^{\x(S)}.$

\item Suppose $d_{S}(a, b)\geq2r+2j+1$, where $j =10\delta +1$. If $c\in g_{a,b}$ and $c \notin N_{r+j}(a)\cup N_{r+j}(b)$, then $|G(a,b;r)\cap N_{2\delta}(c)|\leq K^{\x(S)}.$ 
\end{enumerate}
\end{theorem}

\begin{remark}
We observe that $N_{\delta}(c)$ intersects all (tight) geodesics between $a$ and $b$, and observe that $N_{2\delta}(c)$ intersects all (tight) geodesics between $N_{r}(a)$ and $N_{r}(b)$ by the hyperbolicity of the curve graph.
\end{remark}

We also show that there exist computable bounds for the slices, which depend only on the surface. Our bound will be weaker, yet the proof will be simpler as it is a direct corollary of a special case of Theorem \ref{br}. Also, we note that the hypothesis of the second statement will be weaker, i.e., $j$ will be $3\delta +2$ instead of $10\delta +1$. Furthermore, we also take care of the case when $\x(S)=1.$ Recall $N_{S}(l,k)$ from Theorem \ref{br}. We prove

\begin{theorem} \label{geo}
Suppose $\x(S)\geq1$. Let $a,b\in C(S)$, $r\geq 0$, and $M$ be from Theorem \ref{BGIT}.
\begin{enumerate}
\item  If $c\in g_{a,b}$, then $|G(a,b)\cap N_{\delta}(c)|\leq N_{S}(2M,3).$

\item Suppose $d_{S}(a, b)\geq2r+2j+1$, where $j =3\delta+2$. If $c\in g_{a,b}$ and $c \notin N_{r+j}(a)\cup N_{r+j}(b)$, then $|G(a,b;r)\cap N_{2\delta}(c)|\leq N_{S}(4M,3).$
\end{enumerate}

\end{theorem}

In $\S6$, we define a new class of geodesics, which we call weak tight geodesics. Our bounds in Theorem \ref{geo} are indeed bounds for $M$--weakly tight geodesics. Since tight geodesics are $M$--weakly tight geodesics, see Corollary \ref{2015}, we ask Question \ref{QQ2}, which could fill up some gap between Webb's bounds and our bounds.

Furthermore, we obtain an analog of Theorem \ref{geo} in the setting of weak tight geodesics. We refer the reader to $\S 6$ for the definitions regarding weak tight geodesics and their slices. We show  

\begin{theorem}
\it{Suppose $\x(S)\geq1$. Let $a,b\in C(S)$, $r\geq 0$, and $D\geq M$.}
\begin{enumerate}
\item \it{If $c\in g_{a,b}$, then $ |G^{D}(a,b)\cap N_{\delta}(c)| \leq N_{S}(2D,3).$}

\item \it{Suppose $d_{S}(a, b)\geq2r+2j+1$, where $j =3\delta+2$. If $c\in g_{a,b}$ and $c \notin N_{r+j}(a)\cup N_{r+j}(b)$, then $|G^{D}(a,b;r)\cap N_{2\delta}(c)|\leq N_{S}(2(D+M),3).$}
\end{enumerate}
\end{theorem}

In the rest of this paper, Theorem \ref{br} and Theorem \ref{geo} will be respectively written as Theorem \ref{A} and Theorem \ref{B}.

\renewcommand{\abstractname}{\textbf{Acknowledgements}}
\begin{abstract}
The author thanks Kenneth Bromberg for suggesting to prove Theorem A when $k=3$, which is the key to prove Theorem \ref{B}. The author also thanks the referee for useful suggestions and comments in the revision process, which improved this manuscript significantly.
\end{abstract}

\section{\textbf{Background}}
In this section, we recall the definitions of tight geodesics and subsurface projections from \cite{MM2}. We note that subsurface projections will be a key machinery in this paper.
\subsection{Tight geodesics}
Suppose $A\subseteq S$; we let $S-A$ denote the complementary components of $A$ in $S$, and we treat $S-A$ as embedded subsurfaces in $S$. 

A multicurve is the set of curves that form a simplex in $C(S)$. Suppose $V$ and $W$ are muticurves; we say $V$ and $W$ fill $S$ if there is no curve in $S-(V\cup W)$, i.e., $S-(V\cup W)$ consists of disks and peripheral annuli. Indeed, $V$ and $W$ always fill the subsurface, $S(V,W)$, which is constructed by the following way; take the regular neighborhood of $V \cup W$ and fill in every complementary component of $V \cup W$, which is a disk or a peripheral annulus. We note that this construction has come from \cite{MM2}.

Suppose $A,B\subseteq C(S)$; we define $\d d_{S}(A,B)=\max \{ d_{S}(a,b)| a\in A,b\in B\}$. We observe the following.
\begin{lemma}
Suppose $\x(S)>1$. Let $V$ and $W$ be multicurves. If $d_{S}(V,W)>2$, then $V$ and $W$ fill $S$. 
\end{lemma}

Now, we recall the definition of tight geodesics from $\S 4$ of \cite{MM2}.

\begin{definition}[Tight geodesics]\label{e}
Suppose $\x(S)>1$.
\begin{itemize}
\item A multigeodesic is a sequence of multicurves $V_{0},V_{1},V_{2}\cdots V_{k}$ such that $d_{S}(x,y)=|j-i|$ for all $x\in V_{i}$, $y\in V_{j}$, and $i \neq j$. 
\item A tight multigeodesic is a multigeodesic $V_{0},V_{1},V_{2}\cdots V_{k}$ such that $V_{i}=\partial S(V_{i-1},V_{i+1})$ for all $i$. 
\item A geodesic $v_{0},v_{1},v_{2}\cdots v_{k}$  is a tight geodesic if there exists a tight multigeodesic $V_{0},V_{1},V_{2}\cdots V_{k}$ such that $v_{i}\in V_{i}$ for all $i$. 

\end{itemize}
Suppose $\x(S)=1$. Every geodesic is defined to be a tight geodesic. 

\end{definition}

Masur--Minsky showed
\begin{theorem}[\cite{MM2}]\label{MMMM}
There exists a tight geodesic between any pair of curves in $C(S)$.
\end{theorem}

\subsection{Subsurface projections}
First, we recall close relatives of the curve complex, which are the arc complex $A(S)$, and the arc and curve complex $AC(S)$. In this paper, we assume that arcs are simple and essential. Also, we assume isotopy of arcs is relative to the boundaries setwise unless we say pointwise.

Suppose $\x(S)\geq 0$; the vertices of $A(S)$ ($AC(S)$) are isotopy classes of arcs (arcs and curves) and the simplices of $A(S)$ ($AC(S)$) are collections of arcs (arcs and curves) that can be mutually realized to be disjoint in $S$.

Suppose $x,y\in AC(S)$. The intersection number, $i(x,y)$ is the minimal possible geometric intersections of $x$ and $y$ in their isotopy classes. We say $x$ and $y$ are in minimal position if they realize $i(x,y)$.

Now, we define subsurface projections. There are two types, which are non--annular projections and annular projections. We let $\P(C(S))$ and $\P(AC(S))$ be the set of finite subsets in each complex.

\underline{Non--annular projections:}
Suppose $Z\subseteq S$ such that $Z$ is not an essential annulus. Let $x \in AC(S)$. Assume $x$ and $\partial Z$ are in minimal position. We define the map, $i_{Z}:AC(S)\rightarrow \P(AC(Z))$ by taking isotopy classes of $\{x\cap Z\}$. 

Suppose $A\subseteq S$; we let R(A) denote a regular neighborhood of $A$ in $S$.  We define the map, $p_{Z}:AC(Z)\rightarrow \P(C(Z))$ as follows;
\begin{itemize}
\item if $x\in C(Z)$, then $p_{Z}(x)=x$.
\item if $x\in A(Z)$, take $z,z' \subseteq \partial Z$ such that $\partial(x)$ lie on; then $p_{Z}(x)=\partial R(x\cup z\cup z')$. We note $z$ could be same as $z'$. See Figure \ref{kashiwa}.

\begin{figure}[htbp]
 \begin{center}
\includegraphics[width=120mm]{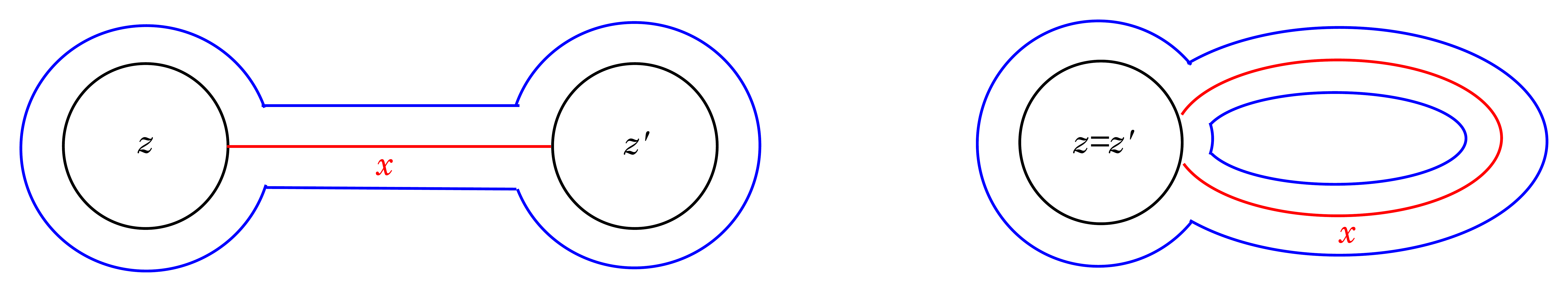}
\end{center}
\caption{$p_{Z}(x)=\partial R(x\cup z\cup z')$. }
 \label{kashiwa}
\end{figure}

\end{itemize}

The subsurface projection to $Z$ is the map, $$\pi_{Z}=p_{Z}\circ i_{Z}:AC(S)\rightarrow \P(C(Z)).$$

\underline{Annular projections:}
Suppose $Z\subseteq S$ such that $Z$ is an essential annulus. Take the annular cover of $S$ which corresponds to $Z$, compactly the cover with $\partial \mathbb{H}^{2}$. We denote the resulting compact annular cover by $S^{Z}$. We first define the annular curve complex, $C(Z)$; the vertices are isotopy classes of arcs whose endpoints lie on two boundaries of $S^{Z}$, here the isotopy is relative to $\partial S^{Z}$ pointwise. The edge between two vertices are realized if they can be isotoped to be disjoint in the interior of $S^{Z}$, again the isotopy is relative to $\partial S^{Z}$ pointwise. By fixing an orientation of $S$ and an ordering the components of $\partial S^{Z}$, algebraic intersection number of $x$ and $y$, $x\cdot y$ is well defined; we observe that $d_{Z}(x,y)=|x\cdot y|+1$. For a detailed treatment, see \cite{MM2}.

Let $x\in AC(S);$ the subsurface projection to $Z$ is the map, $$\pi_{Z}:AC(S)\rightarrow \P(C(Z))$$ such that $\pi_{Z}(x)$ is the set of all arcs connecting two boundaries of $S^{Z}$, which are obtained by the lift of $x$. We observe that $\pi_{Z}(x)=\emptyset$ if and only if $i(\p Z,x)=0$.

For both types of projections, if $A\subseteq AC(S)$ and $Z\subseteq S$, then we define $\d \pi_{Z}(A)=\bigcup_{a\in A} \pi_{Z}(a)$.

We observe the following by the definition of subsurface projections. Throughout in this paper, we often use it without referring.
\begin{lemma}\label{distance3}
Suppose $Z\subseteq S$ and $A,B\subseteq C(S)$. If $A$ and $B$ fill $S$, then $A$ or $B$ projects nontrivially to $Z$. In particular, if $d_{S}(A,B)>2$, then we always have $$ \pi_{Z}(A)\neq \emptyset \text{ or } \pi_{Z}(B)\neq \emptyset.$$
\end{lemma}

\subsubsection{Some results on subsurface projections}
Suppose $A,B \subseteq AC(S)$; we define $d_{Z}(A,B)$ to be the diameter of $\pi_{Z}(A) \cup \pi_{Z}(B)$ in $C(Z).$ Recall

\begin{definition}[Dehn twist]
Suppose $Z$ is an essential annulus in $S$. Let the core curve of $Z$ be $x \in C(S)$. We topologically understand $Z$ by $\{x\}\times [0,1]=S^{1}\times [0,1]$; Dehn twist along $x$, $T_{x}$ is defined as follows:
$$T_{x}(a) =
\left\{
	\begin{array}{ll}
		a  & \mbox{if } a\notin Z \\
		 \big(e^{2i\pi(\theta+r)}, r \big)& \mbox{if } a=(e^{2i\pi(\theta)}, r)\in Z=S^{1}\times [0,1]
	\end{array}
\right.$$
\end{definition}

First, we observe

\begin{lemma} [\cite{MM2}]\label{min}
Suppose $Z$ is an essential annulus in $S$. Let $x\in C(S)$ be the core curve of $Z$ and let $T_{x}$ be Dehn twist of $x$. If $y\in C(S)$ is such that $\pi_{Z}(y)\neq \emptyset$, then $$d_{Z}(y,T_{x}^{n}(y))=|n|+2 \text{ for }n\neq0.$$

If $y$ intersects $x$ exactly twice with opposite orientation, then a half twist to $y$ is well defined to obtain a curve $H_{x}(y)$, which is taking $x\cup y$ and resolving the intersections in a way consistent with the orientation. Then $H_{x}^{2}(y)=T_{x}(y)$, and  $$\d d_{Z}(y,H_{x}^{n}(y))=\left\lfloor\dfrac{|n|}{2}\right\rfloor+2  \text{ for }n\neq0.$$
\end{lemma}

Lastly, we observe the Bounded Geodesic Image Theorem; it was first proved by Masur--Minsky \cite{MM2}. A recent work of Webb shows that the bound depends only on the hyperbolicity constant, which implies that the bound is computable and uniform for all surfaces \cite{WEB1}. Here, we state Webb's version of the Bounded Geodesic Image Theorem.

\begin{theorem} [Bounded Geodesic Image Theorem]\label{BGIT}
There exists $M(\delta)$ such that the following holds. If $\{x_{i}\}_{0}^{n}$ is a multigeodesic such that $x_{i}$ projects nontrivially to $Z\subsetneq S$ for all $i$, then $$d_{Z} (x_{0},x_{n}) \leq M(\delta).$$
\end{theorem}
   
In the rest of this paper, we mean $M$ as $M$ in the statement of the Bounded Geodesic Image Theorem.

\section{\textbf{Outline}}
The goal of this section is to capture the main contents of this paper with a more detailed discussion than $\S 1$. 

First, we rephrase Theorem \ref{br} by Theorem \ref{A} by using the following definition.
\begin{definition}\label{defintion}
Suppose $\x(S)\geq 1$ and $A\subseteq C(S)$. Let $l>0$, $k>1$ and $Z\subseteq S$. 
\begin{itemize}
\item We say $A$ satisfies the property $\P(l,k,Z)$ if $A$ does not contain $k$ curves whose projections to $Z$ are mutually more than $l$ apart in $C(Z)$. 

\item We say $A$ satisfies the property $\P(l,k)$ if $A$ satisfies the property $\P(l,k,Z)$ for all $Z\subseteq S$.
\end{itemize}
\end{definition}
\begin{remark}
For some readers, the following equivalent definition could be more convenient; 
we say $A$ satisfies the property $\P(l,k,Z)$ if there is no $A'\subseteq A$ such that $|A'|= k$ and $d_{Z}(x,y)>l$ for all $x,y\in A'$. 
\end{remark}

First, we prove the following main theorem, $\U$ of the curve graph via subsurface projections.
\begin{thmx}[Contrapositive version of Theorem \ref{br}]\label{A}
\it{Suppose $\x(S)\geq 1$. There exists a computable $N_{X}(l,k)$ for any $l>0$ and $k>1$ such that the following holds. If $A\subseteq C(S)$ such that $A$ satisfies $\P(l,k)$, then $|A|\leq N_{S}(l,k)$.}
\end{thmx}

We prove Theorem \ref{A} by a double induction of the complexity and the distance. See $\S 5$.

By using Theorem \ref{A} with $k=3$, we show 

\begin{thmx}[Theorem \ref{geo}]\label{B}
\it{Suppose $\x(S)\geq1$. Let $a,b\in C(S)$ and $r\geq 0$.}
\begin{enumerate}
\item \it{If $c\in g_{a,b}$, then $ |G(a,b)\cap N_{\delta}(c)| \leq N_{S}(2M,3).$}

\item \it{Suppose $d_{S}(a, b)\geq2r+2j+1$, where $j =3\delta+2$. If $c\in g_{a,b}$ and $c \notin N_{r+j}(a)\cup N_{r+j}(b)$, then $|G(a,b;r)\cap N_{2\delta}(c)|\leq N_{S}(4M,3).$}
\end{enumerate}
\end{thmx}

For the second statement of Theorem \ref{B}, we have to argue a bit, yet the proof mostly consists of technical arguments which are commonly used in $\delta$--hyperbolic spaces. See Corollary \ref{c}. Once we have Corollary \ref{c}, we can easily observe Lemma \ref{cno3}. Combining Theorem \ref{A} with Lemma \ref{cno3}, we obtain Theorem \ref{B}.

We note that a key machinery for the proof of Corollary \ref{c} is Lemma \ref{ti}, which observes an interesting behavior of tight geodesics with the Bounded Geodesic Image Theorem. 
Indeed, Lemma \ref{ti} was the motivation for the author to define weak tight geodesics. See $\S6$, where we also prove Theorem \ref{B} in the setting of weak tight geodesics; for a specific statement, see Theorem \ref{weaktight}.

\section{\textbf{Theorem \ref{A} with $k=3$ implies Theorem \ref{B}}}

We observe the following important property of tight geodesics; this is a key observation which we use to prove Corollary \ref{c}. Suppose $x,y\in C(S)$; we let $g_{x,y}^{t}$ be a tight geodesic between $x$ and $y$. 

\begin{lemma}\label{ti}
Suppose $\x(S)\geq 1$ and $Z\subsetneq S$. Let $x,y\in C(S)$ and $v\in g_{x,y}^{t}$. If $\pi_{Z}(v)\neq \emptyset$, then $$d_{Z}(x,v)\leq M \text{ or } d_{Z}(v,y)\leq M.$$
\end{lemma}
\begin{proof} 
Let $g_{x,y}^{t}=\{v_{i}\}$ such that $d_{S}(x,v_{i})=i$ for all $i$. Assume $v=v_{k}.$ Take a tight multigeodesic $\{V_{i}\}$ such that $v_{i}\in V_{i}$ for all $i$. 

Suppose $\pi_{Z}(V_{i}) \neq \emptyset$ for all $i>k$. Then, by Theorem \ref{BGIT}, we obtain $ d_{Z}(v_{k},y)\leq M .$

Suppose not, then we use Lemma 4.10 of \cite{MM2}, which states that the vertices of a tight multigeodesic that miss to project to $Z$ (called the footprint in \cite{MM2}) is a subsequence of 0, 1, 2, or 3 contiguous vertices. Therefore, we must have $\pi_{Z}(V_{i}) \neq \emptyset$ for all $i<k$; again by Theorem \ref{BGIT}, we obtain $d_{Z}(x,v_{k})\leq M.$
 \end{proof}

The following is a corollary of Lemma \ref{ti} with some technical observations.

\begin{corollary}\label{c}
Suppose $\x(S)\geq1$ and $Z\subsetneq S$. Let $a,b\in C(S)$ and $r\geq 0$.
\begin{enumerate}
\item Let $c\in g_{a,b}$. If $x\in G(a,b)\cap N_{\delta}(c)$ and $\pi_{Z}(x) \neq \emptyset$, then $$d_{Z}(a,x)\leq M\text{ or }d_{Z}(x,b)\leq M.$$ 

\item Suppose $d_{S}(a, b)\geq2r+2j+1$, where $j =3\delta+2$. Let $c\in g_{a,b}$ such that $c \notin N_{r+j}(a)\cup N_{r+j}(b)$. If $x\in G(a,b;r)\cap N_{2\delta}(c)$ and $\pi_{Z}(x) \neq \emptyset$, then $$d_{Z}(a,x)\leq 2M\text{ or }d_{Z}(x,b)\leq 2M.$$ 
\end{enumerate}
\end{corollary}
\begin{proof}
The first statement follows from Lemma \ref{ti}.

We prove the second statement.
Let $A\in N_{r}(a)$ and $B\in N_{r}(b)$ such that $x$ is contained in $g_{A,B}^{t}$. Let $g_{A,x}^{t}$ be the subsegment of $g_{A,B}^{t}$ connecting $A$ and $x$, and let $g_{x,B}^{t}$ be the subsegment of $g_{A,B}^{t}$ connecting $x$ and $B$. See Figure \ref{Fig1}.

\begin{figure}[htbp]
 \begin{center}
  \includegraphics[width=100mm]{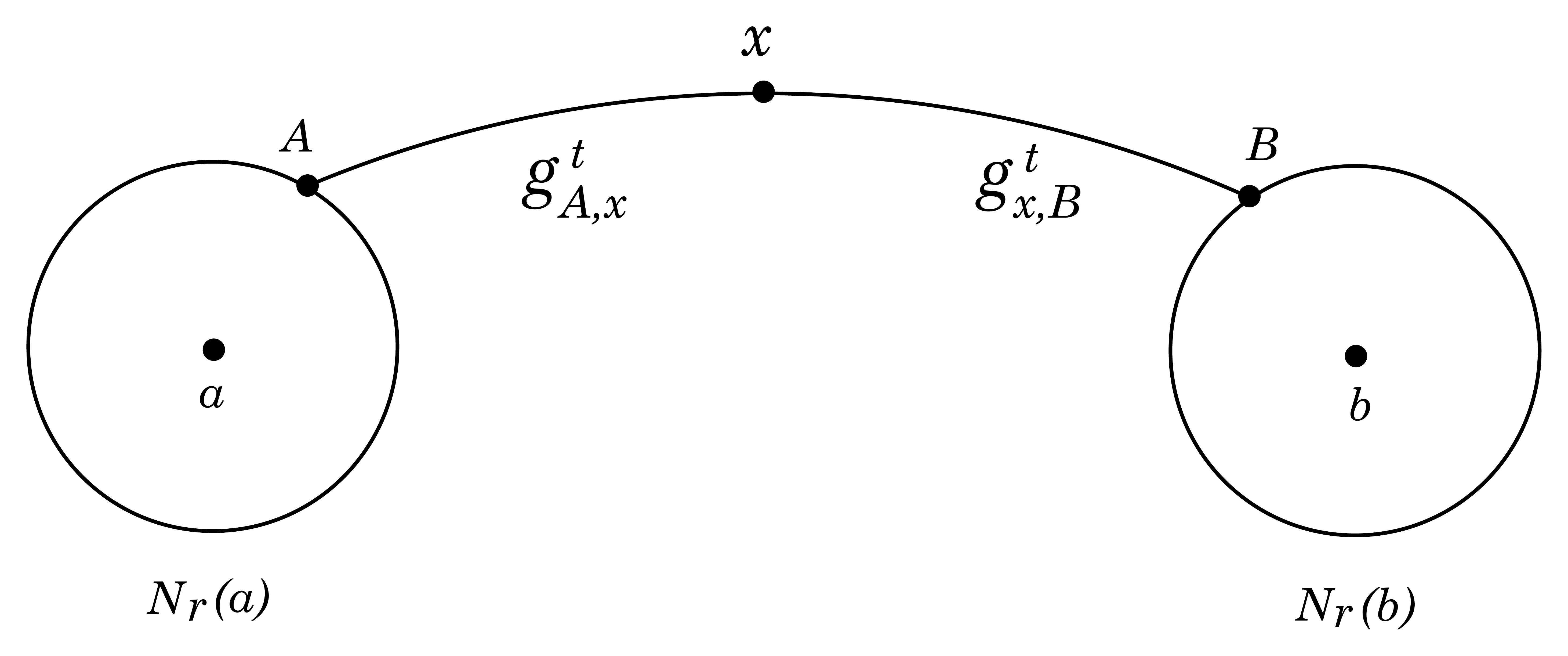}
 \end{center}
 \caption{$g_{A,B}^{t}= g_{A,x}^{t}\cup g_{x,B}^{t}$. }
 \label{Fig1}
\end{figure}

By Lemma \ref{ti}, we have $$d_{Z}(A,x)\leq M \text{ or } d_{Z}(x,B)\leq M.$$
Therefore, it suffices to consider the following two cases. 
\begin{itemize}
\item \underline{Case 1} is when we have $d_{Z}(A,x)\leq M$ and $d_{Z}(x,B)\leq M$. 
\item \underline{Case 2} is when we have either $d_{Z}(A,x)\leq M$ or $d_{Z}(x,B)\leq M$. 
\end{itemize}

\underline{Case 1}: We claim $$ d_{Z}(a,A)\leq M \text{ or } d_{Z}(B, b)\leq M.$$

 With this claim, we have $$\text{ }d_{Z}(a,x)\leq 2M \text{ or } d_{Z}(x, b)\leq 2M.$$

\emph{Proof of the claim.}
Since $d_{S}(a, b)\geq2r+2j+1$, we observe that $N_{r}(a)$ and $N_{r}(b)$ are far apart, i.e., $$\d \min_{a'\in N_{r}(a), b'\in N_{r}(b)} d_{S}(a',b')>2.$$ Recall that $g_{a,A}\subseteq N_{r}(a)$ and $g_{B,b}\subseteq N_{r}(b);$ by the above observation and Lemma \ref{distance3}, we observe that  every vertex of $g_{a,A}$ or $g_{B,b}$ projects nontrivially to $Z$. Therefore, by Theorem \ref{BGIT}, we have $$d_{Z}(a,A)\leq M \text{ or } d_{Z}(B, b)\leq M.$$

\underline{Case 2}: Without loss of generality, we assume that $$d_{Z}(A,x)\leq M \text{ and }d_{Z}(x,B)>M.$$Since $d_{Z}(x,B)>M$, there exists $q\in g_{x,B}^{t}$ such that $\pi_{Z}(q)= \emptyset$ by Theorem \ref{BGIT}.

We claim $$q\in N_{ d_{S}(c,b)+3\delta}(b).$$ 

First, we show that the claim implies the statement of this corollary. By the hypothesis on $c$, we have $N_{r+2}(a)\cap N_{ d_{S}(c,b)+3\delta}(b) = \emptyset.$
Therefore, by the claim, we can conclude $q\notin N_{r+2}(a).$ Now, since $g_{a,A}\subseteq N_{r}(a)$, we observe that $d_{S}(p,q)>2$ for all $ p\in g_{a,A}$, i.e., every vertex of $g_{a,A}$ and $q$ fill $S$. Therefore, every vertex of $g_{a,A}$ projects nontrivially to $Z$; by Theorem \ref{BGIT}, we have $$d_{Z}(a,A)\leq M. \Longrightarrow d_{Z}(a,x)\leq d_{Z}(a,A)+d_{Z}(A,x)\leq 2M.$$

\emph{Proof of the claim.}
Recall the claim; $$q\in N_{ d_{S}(c,b)+3\delta}(b). \Longleftrightarrow d_{S}(q, b)\leq d_{S}(c,b)+3\delta.$$ 
The proof will be done by a standard argument often used in $\delta$--hyperbolic spaces.
Consider the $4$--gon whose edges are $g_{x,B}^{t}, g_{B,b}, g_{b,c}$, and $g_{c,x}$. Take an additional geodesic $g_{c,B}$, which decomposes the $4$--gon into two triangles. See Figure \ref{Fig2}. 

\begin{figure}[htbp]
 \begin{center}
  \includegraphics[width=100mm]{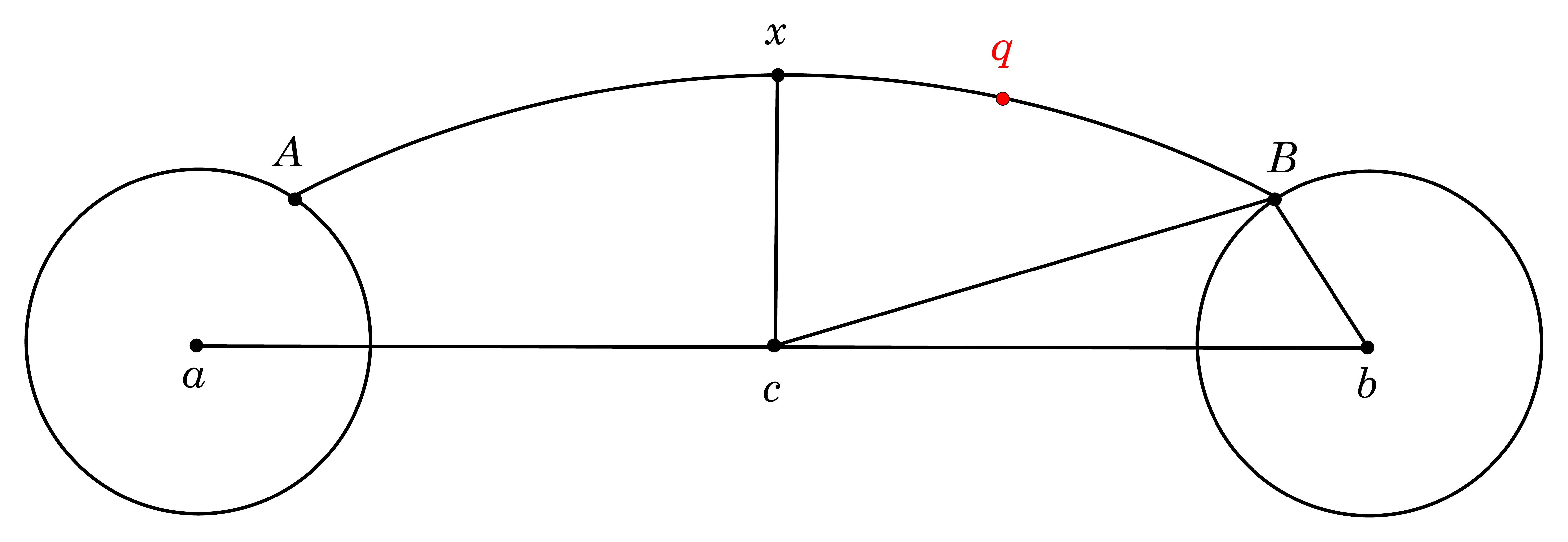}
 \end{center}
 \caption{The $4$--gon with $g_{c,B}$}
 \label{Fig2}
\end{figure}

By the hyperbolicity, there exists $\g \in g_{x,c}\cup g_{c,B}$ such that $d_{S}(q,\g)\leq \delta$.

If $\g \in g_{x,c} $, then 
\begin{align}
d_{S}(q, b)&\leq d_{S}(q,\g)+d_{S}(\g, c)+d_{S}(c, b) \tag{$\triangle$--inequality}
\\&\leq \delta +d_{S}(\g, c)+d_{S}(c, b) \tag{Since $d_{S}(q,\g)\leq \delta$}
\\&\leq 3 \delta +d_{S}(c, b). \tag{Since $d_{S}(\g, c)\leq d_{S}(x,c)\leq 2\delta$}
\end{align}

If $\g \in g_{c,B} $, then we use the hyperbolicity again on the triangle $\triangle_{c,B,b}$; there exists $\g' \in g_{c,b}\cup g_{b,B}$ such that $d_{S}(\g,\g')\leq \delta.$ Therefore, we have $$d_{S}(q,\g') \leq d_{S}(q,\g)+ d_{S}(\g,\g')\leq 2\delta.$$ If $\g' \in g_{c,b},$ then $$d_{S}(q, b)\leq d_{S}(q, \g')+d_{S}(\g', b)\leq 2\delta+ d_{S}(c, b)\leq 3 \delta +d_{S}(c, b).$$
If $\g' \in g_{b,B},$ then $$d_{S}(q, b)\leq d_{S}(q, \g')+d_{S}(\g', b) \leq 2\delta+ d_{S}(B,b)\leq  3 \delta +d_{S}(c, b).$$

\end{proof}

Now, we obtain the goal of this section; Theorem \ref{A} with $k=3$ implies Theorem \ref{B}.
First, we observe the following lemma which directly follows from Definition \ref{defintion}.

\begin{lemma}\label{bbb}
Suppose $\x(S)\geq1,$ $Z\subseteq S$, and $A\subseteq C(S)$. If $A$ satisfies $\P(l,k,Z)$, then $A$ satisfies $\P(l',k',Z)$ whenever $l'\geq l$ and $k' \geq k$.
\end{lemma}

By Corollary \ref{c}, we have 
\begin{lemma}\label{cno3}
Suppose $\x(S)\geq1$. Let $a,b\in C(S)$ and $r\geq 0$.
\begin{enumerate}
\item If $c\in g_{a,b}$, then $G(a,b)\cap N_{\delta}(c)$ satisfies $\P(2M,3)$.
\item Suppose $d_{S}(a, b)\geq2r+2j+1$, where $j =3\delta+2$. If $c\in g_{a,b}$ and $c \notin N_{r+j}(a)\cup N_{r+j}(b)$, then $G(a,b;r)\cap N_{2\delta}(c)$ satisfies $\P(4M,3)$.
\end{enumerate}
\end{lemma}
\begin{proof}
We only prove the first statement; the same proof works for the second statement.

First, we observe that $G(a,b)\cap N_{\delta}(c)$ satisfies $\P(2 \delta ,2,S)$. Since $M>4\delta$ \cite{WEB1}, with Lemma \ref{bbb}, we observe that $G(a,b)\cap N_{\delta}(c)$ satisfies $\P(2M,3,S)$.

Now, we show that $G(a,b)\cap N_{\delta}(c)$ satisfies $\P(2M,3,Z)$ for all $Z\subsetneq S$. Suppose $y\in C(Z)$ and $p\geq 0$; we define $$N_{p}^{C(Z)}(y)=\{w\in C(Z)| d_{Z}(y,w)\leq p\}.$$
Corollary \ref{c} states that $$ \pi_{Z}(G(a,b)\cap N_{\delta}(c))           \subseteq  N_{M}^{C(Z)}(  \pi_{Z}(a)) \cup N_{M}^{C(Z)}(  \pi_{Z}(b)).$$

Since the projection of every element of $G(a,b)\cap N_{\delta}(c)$ to $Z$ is contained in $N_{M}^{C(Z)}(  \pi_{Z}(a))$ or $N_{M}^{C(Z)}(  \pi_{Z}(b)),$ which are the balls of diameter $2M$ in $C(Z)$ (see Figure \ref{1000}), we conclude that $G(a,b)\cap N_{\delta}(c)$ does not contain ``$3$'' elements whose projections to $Z$ are mutually more than ``$2M$'' apart in $C(Z)$. 

\begin{figure}[htbp]
 \begin{center}
  \includegraphics[width=90mm]{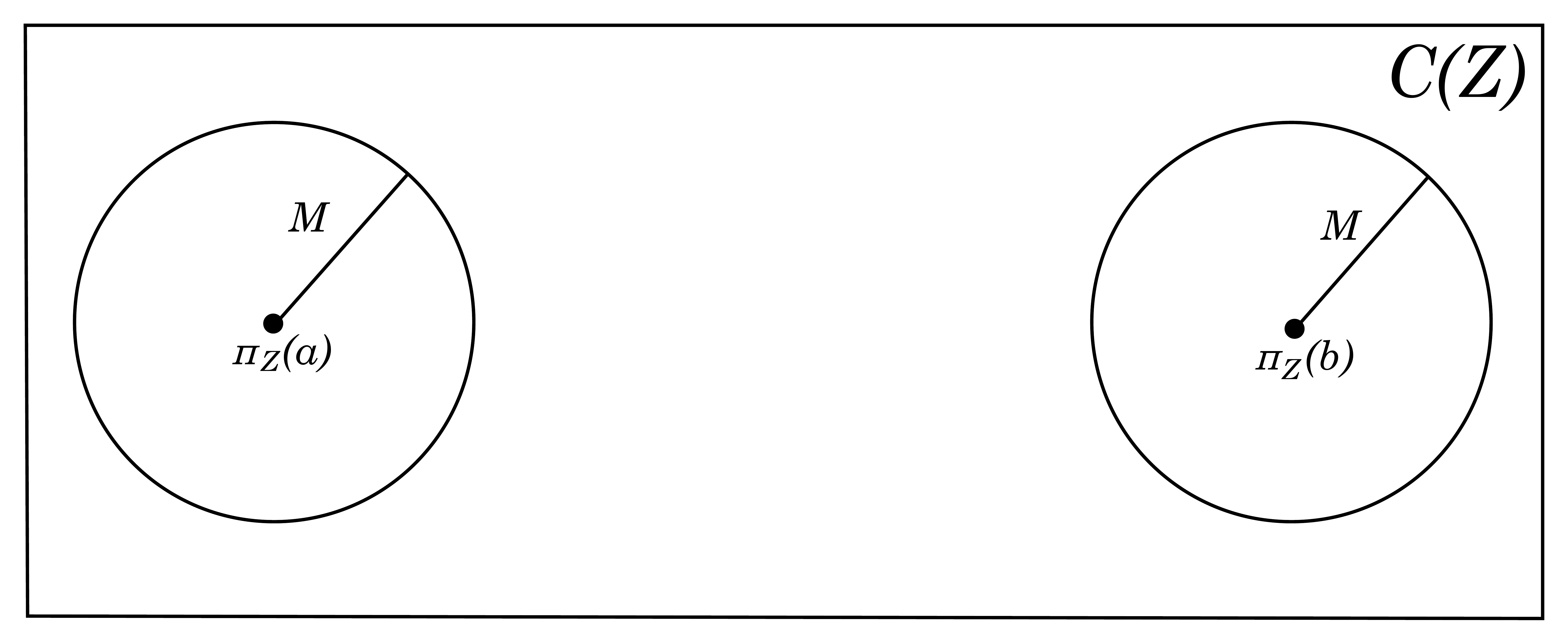}
 \end{center}
 \caption{$N_{M}^{C(Z)}(  \pi_{Z}(a))$ and $N_{M}^{C(Z)}(  \pi_{Z}(b))$ in $C(Z)$.}
 \label{1000}
\end{figure}

\end{proof}

\begin{remark}
By Lemma \ref{cno3}, we observe that Theorem \ref{B} follows from Theorem \ref{A} with $k=3$.
\end{remark}

\section{\textbf{The proof of Theorem \ref{A}}}
We prove Theorem \ref{A} by a double induction on the complexity and the distance. We first show for $S_{1,1}$ and $S_{0,4}$; the curve graphs of them are Farey graphs whose vertices are identified with $\mathbb{Q}\cup \{\frac{1}{0}=\infty\}$. See \cite{FAR}.
Suppose $x\in C(S)$; we define $C_{i}(x)=\{y\in C(S)|d_{S}(x,y)=i\}.$ 




We observe the following lemma which is the heart of the proof of Theorem \ref{A}.

\begin{lemma} \label{co} Let $x\in C(S)$ and $B\subseteq C_{i}(x)$ for $i>1$.

\begin{itemize}
\item Suppose $\x(S)=1$. Let $Z=R(x)$, i.e., $Z$ is the annulus whose core curve is $x$. If $B$ satisfies $\P(l,k,Z)$, then there exist $B' \subseteq C_{1}(x)$ such that $\d B\subseteq \bigcup_{y \in B'}C_{i-1}(y)$ and $B' $ satisfies $\P(l+2M,k,Z)$.

\item Suppose $\x(S)>1$. Let $Z \subseteq S-x$. If $B$ satisfies $\P(l,k,Z)$, then there exist $B' \subseteq C_{1}(x)$ such that $\d B\subseteq \bigcup_{y \in B'} C_{i-1}(y)$ and $B'$ satisfies $\P(l+2M,k,Z)$.

\end{itemize}
\end{lemma}

\begin{proof} The proof will be the combination of tightness, Lemma \ref{ti} (for $\x(S)>1$) and the Bounded Geodesic Image Theorem, Theorem \ref{BGIT}.

\underline{Suppose $\x(S)=1$}: Let $b\in C_{i}(x)$; we observe that every vertex of $g_{x,b}\setminus \{x\}$ projects nontrivially to $R(x)$. Therefore, by letting $b'=g_{x,b}\cap C_{1}(x)$, we have $$d_{R(x)}(b',b)\leq M \text{ by Theorem \ref{BGIT}}.$$ Furthermore, let $c\in C_{i}(x)$ and $c'=g_{x,c}\cap C_{1}(x)$, then we have 
\begin{eqnarray*}
d_{R(x)}(b',c')&\leq& d_{R(x)}(b',b)+d_{R(x)}(b,c)+d_{R(x)}(c,c')\\&\leq& d_{R(x)}(b,c)+2M.
\end{eqnarray*}
See Figure \ref{Fig3}. 

\begin{figure}[htbp]
 \begin{center}
  \includegraphics[width=50mm]{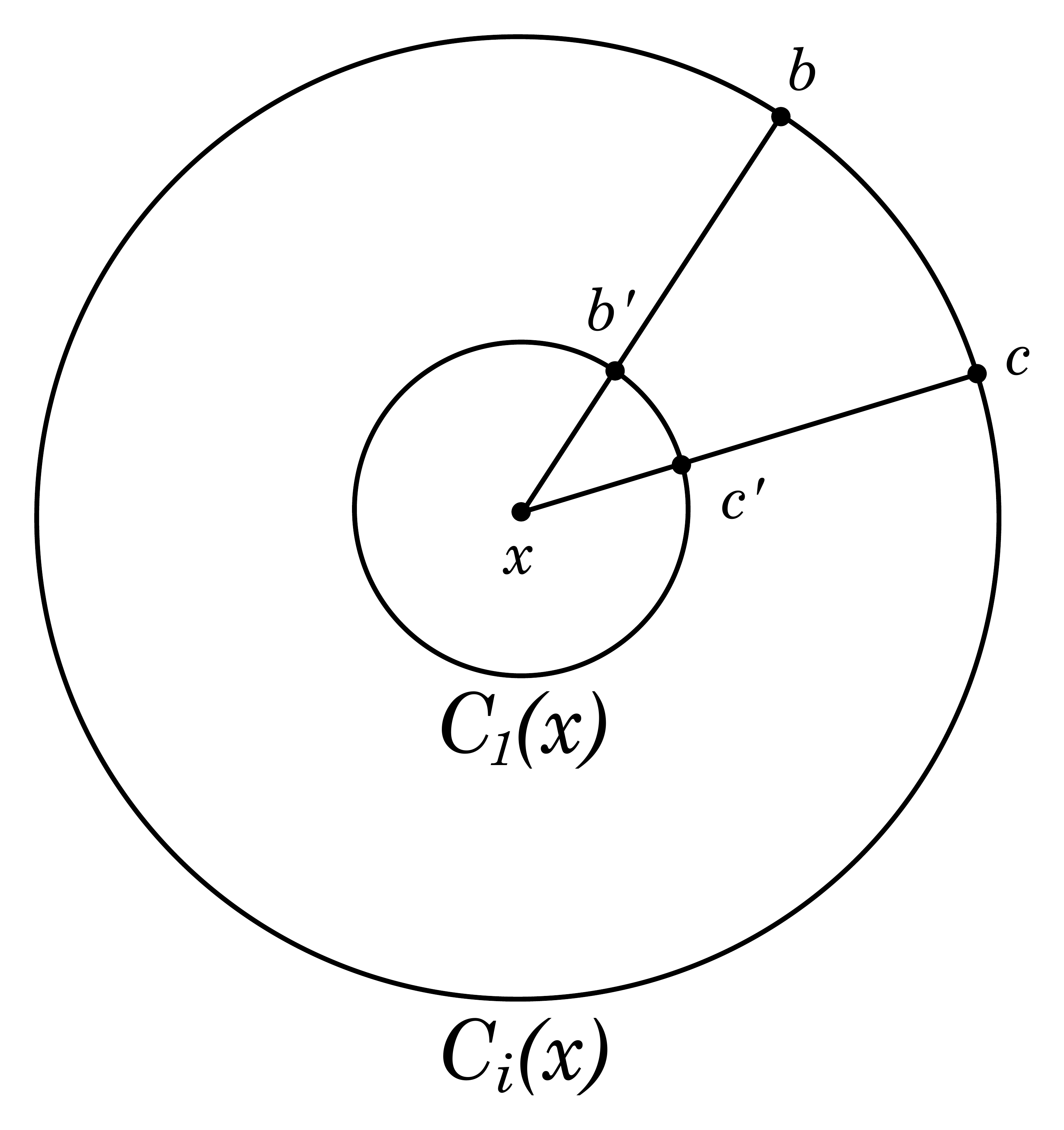}
 \end{center}
 \caption{$d_{R(x)}(b',c')\leq d_{R(x)}(b,c)+2M.$}
 \label{Fig3}
\end{figure}

Define $\d B'=\bigcup_{w\in B} g_{x,w}\cap C_{1}(x).$ Since $B$ satisfies $\P(l,k,Z)$, by the above observation, we conclude that $B'$ satisfies $\P(l+2M,k,Z)$. Lastly, by the definition of $B'$, we observe $\d B\subseteq \bigcup_{y \in B'} C_{i-1}(y).$

\underline{Suppose $\x(S)>1$}: The proof is analogous to the proof of the previous case; we need to take tight geodesics instead of geodesics. 

Let $b\in C_{i}(x)$, and take $g_{x,b}^{t}$; if $b'= g_{x,b}^{t}\cap C_{1}(x)$ projects nontrivially to $Z \subseteq S-x$, then since $\pi_{Z}(x)=\emptyset$, we have $$d_{Z}(b',b)\leq M \text{ by tightness and Theorem \ref{BGIT}}.$$

As in the previous case, if we define $\d B'=\bigcup_{w\in B} g_{x,w}^{t}\cap C_{1}(x)$, then $B'$ satisfies $\P(l+2M,k,Z)$ and we also have $\d B\subseteq \bigcup_{y \in B'} C_{i-1}(y).$
\end{proof}

Now, we prove Theorem \ref{A} for $\x(S)=1$. For simplicity, we prove it for $S=S_{1,1}.$

\begin{theorem}\label{1}
Suppose $S=S_{1,1}$ and $A\subseteq C(S)$. Let $l>0$ and $k>1$. If $A$ satisfies $\P(l,k)$, then $|A|\leq (Lk)^{l+1}\text{ where }L=l+2M+2.$

\end{theorem}
\begin{proof}

Since $A$ satisfies $\P(l,k,S)$, $A$ does not contain $k$ curves which are mutually more than $l$ apart in $C(S)$; which means we can pick $\{x_{i}\}_{i=1}^{k-1}\subseteq C(S)$ such that $$\d A\subseteq \bigcup_{1\leq i\leq k-1} N_{l}(x_{i}). \Longrightarrow  \d |A |\leq \sum_{i=1}^{k-1}  |A \cap N_{l}(x_{i})|.$$
Therefore, it suffices to understand a bound for $A \cap N_{l}(x)$ where $x\in C(S)$.

Let $L=l+2M+2$. We claim $$|A\cap C_{i}(x)|\leq (Lk)^{i}\text{ for all }i \leq l.$$ By this claim, we have 
\begin{eqnarray*}
|A |\leq (k-1)\cdot |A \cap N_{l}(x)|&=& (k-1)\cdot \Bigg|\bigcup_{0\leq i\leq l}A \cap C_{i}(x) \Bigg|\\\\&=&(k-1)\cdot \bigg( \sum_{i=0}^{l} |A\cap C_{i}(x)| \bigg)\\\\&\leq& (k-1)\cdot  \bigg(\sum_{i=0}^{l} (Lk)^{i}  \bigg)  \\\\&=& (k-1)\cdot \bigg( \frac{(Lk)^{l+1}-1}{Lk-1} \bigg) \\\\&\leq& (Lk)^{l+1}.
\end{eqnarray*}

\emph{Proof of the claim.} We prove it by the induction on $i$.

\underline{Base case, $i=1$}: We recall that if $\frac{s}{t},\frac{p}{q}\in C(S_{1,1}),$ then $i \big(\frac{s}{t},\frac{p}{q} \big)=|sq-tp|.$ We may assume $x=\frac{1}{0}$, then $C_{1}(x)=\mathbb{Z}.$
Let $T_{x}$ be Dehn twist along $x$; by Lemma \ref{min}, if $y\in C_{1}(x)$, then we have $$d_{R(x)}(T_{x}^{i}(y),y)=|i|+2.$$ We also notice that $$C_{1}(x)=\mathbb{Z}=\{T_{x}^{i}(y) \}_{i\in \mathbb{Z}}.$$ Therefore, as in Example \ref{firstsection}, we have $$|A\cap C_{1}(x)|\leq (l+2)k.$$

\underline{Inductive step, $i>1$}: Let $B=A\cap C_{i}(x)$. By our hypothesis, $A$ satisfies $\P(l,k)$; in particular, $A$ satisfies $\P(l,k,R(x))$ since $R(x)\subseteq S$. Lastly, because $B\subseteq A$, we observe that $$\text{$B$ satisfies $\P(l,k,R(x))$.}$$

Now, we use Lemma \ref{co}; since $B$ satisfies $\P(l,k,R(x))$, there exists $B' \subseteq C_{1}(x)$ such that $\d B\subseteq \bigcup_{y\in B'} C_{i-1}(y)$ and $B'$ satisfies $\P(l+2M,k,R(x))$.
Since $B\subseteq A$, we observe $$\d B\subseteq \bigcup_{y\in B'} \big( A \cap C_{i-1}(y) \big). \Longrightarrow  |B| \leq |B'| \cdot | A\cap C_{i-1}(y)|.$$
Here, we run the induction. By the base case, we have $$|B'|\leq (l+2M+2)k\leq Lk.$$ With our inductive hypothesis, we have $$|B| \leq |B'| \cdot | A\cap C_{i-1}(y)| \leq (Lk) \cdot(Lk)^{i-1}\leq (Lk)^{i}.$$
\end{proof}
\begin{remark}
For $S_{0,4}$, the same argument works; the only difference is that we use $H_{x}$ (the half twist along $x$) instead of $T_{x}$ on the base case.  Assuming the same setting in Theorem \ref{1}, we have $$|A\cap C_{i}(x)|\leq (2Lk)^{i} \text{ for all }i \leq l. \Longrightarrow N_{S_{0,4}}(l,k)=(2Lk)^{l+1}.$$
\end{remark}

Now, we complete the proof of Theorem \ref{A}. The nature of the proof is similar to the above proof; we define $$\d N_{S'}(l,k)=\max \{ N_{S_{g,n}}(l,k)| \x(S_{g,n})<\x(S)\}.$$
\begin{theorem}
Suppose $\x(S)> 1$ and $A\subseteq C(S)$. Let $l>0$ and $k>1$. If $A$ satisfies $\P(l,k)$, then $|A|\leq (2N_{S'}(L,k))^{l+1} \text{ where }L=l+2M.$
\end{theorem}

\begin{proof}

Since $A$ satisfies $\P(l,k,S)$, it suffices to understand a bound for $A \cap N_{l}(x)$ where $x\in C(S)$.

Let $L=l+2M$. We claim $$|A\cap C_{i}(x)|\leq (2N_{S'}(L,k))^{i}\text{ for all }i \leq l.$$ By this claim, we have 
\begin{eqnarray*}
|A |\leq (k-1)\cdot |A \cap N_{l}(x)| &=&(k-1)\cdot \bigg( \sum_{i=0}^{l} |A\cap C_{i}(x)| \bigg)\\\\&\leq& (k-1)\cdot \bigg(\sum_{i=0}^{l}  (2N_{S'}(L,k))^{i}  \bigg) \\\\&=& (k-1)\cdot \bigg( \frac{(2N_{S'}(L,k))^{l+1}-1}{2N_{S'}(L,k)-1} \bigg) \\\\&\leq &(2N_{S'}(L,k))^{l+1}.
\end{eqnarray*}

\emph{Proof of the claim.} We prove it by the induction on $i$.

\underline{Base case, $i=1$}: Let $S-x=\{S_{1},S_{2}\}$. We may assume $\x(S_{1}),\x(S_{2})\geq 1$ since $C(S_{0,3})=\emptyset$. Note that $S-x$ may have only one component. 

We may think of the elements in $A\cap C_{1}(x)$ as the elements in $C(S_{1}) \cup C(S_{2})$. Therefore, we have 
$$|A\cap C_{1}(x)|=|\big( A\cap C_{1}(x) \big) \cap C(S_{1})|+ |\big( A\cap C_{1}(x) \big) \cap C(S_{2})|.$$
Since $A\cap C_{1}(x)$ satisfies $\P(l,k)$, we observe that for both $i\in \{1,2\}$, $\big( A\cap C_{1}(x) \big) \cap C(S_{i})$ satisfies $\P(l,k,Z)$ for all $Z\subseteq S_{i}$. Therefore, by using the inductive hypothesis on the complexity, we have $$|\big( A\cap C_{1}(x) \big) \cap C(S_{i})|\leq N_{S_{i}}(l,k) \text{ for both } i\in \{1,2\}.$$
Hence, we have $$|A\cap C_{1}(x)|\leq N_{S_{1}}(l,k)+N_{S_{2}}(l,k)\leq 2N_{S'}(l,k) .$$


\underline{Inductive step, $i>1$}: Let $B=A\cap C_{i}(x)$. Since $B$ satisfies $\P(l,k,Z)$ for all $Z\subseteq S-x$, by Lemma \ref{co}, there exists $B'\subseteq C_{1}(x)$ such that $\d B\subseteq \bigcup_{y\in B'} C_{i-1}(y)$ and $B'$ satisfies $\P(l+2M,k,Z)$ for all $Z\subseteq S-x$. Therefore, we have 
\begin{eqnarray*}
|B|\leq |B'| \cdot | A\cap C_{i-1}(y)|&\leq& 2N_{S'}(l+2M,k) \cdot(2N_{S'}(L,k))^{i-1} \\\\&=&(2N_{S'}(L,k))^{i}.
\end{eqnarray*}

\end{proof}

We understand the growth of $N_{S}(l,k)$.

\begin{eqnarray*}
N_{S}(l,k)= (2N_{S'}(L,k))^{l+1} &=& 2^{l+1} \cdot N_{S'}(L,k)^{l+1}\\\\&\leq &N_{S'}(L,k)^{2(l+1)}\\\\&\leq &N_{S'}(L,k)^{2L}.
\end{eqnarray*}

One can check that 
\begin{eqnarray*}
N_{S'}(L,k)^{2L} &\leq& N_{S_{0,4}}(l+\x(S)\cdot 2M,k )^{ \bigg(  \big(2\cdot(l+\x(S)\cdot 2M)\big) ^{\x(S)}\bigg)} \\\\&\leq & N_{S_{0,4}}(\x(S)\cdot L,k )^{ \big(  (2\cdot \x(S)\cdot L)) ^{\x(S)}\big)}.
\end{eqnarray*}

\section{Weak tight geodesics}
We define a new class of geodesics; this definition is motivated by Lemma \ref{ti}. 
\begin{definition}[Weak tight geodesics]\label{dokoni}
Suppose $\x(S)\geq 1.$ Let $x,y \in C(S)$ such that $d_{S}(x,y)>2$. We say a geodesic $g_{x,y}$ is a $D$--weakly tight geodesic if the following holds for all $v\in g_{x,y}$ and $Z\subsetneq S$. If $\pi_{Z}(v)\neq \emptyset$, then $$d_{Z}(x,v)\leq D \text{ or }d_{Z}(v,y)\leq D.$$
\end{definition}

We remark the following regarding Definition \ref{dokoni}. 
\begin{remark}\label{2016}
\begin{itemize}

\item We need to assume $d_{S}(x,y)>2$, so that $\pi_{Z}(x)\neq \emptyset \text{ or }\pi_{Z}(y)\neq \emptyset \text{ for all }Z\subsetneq S.$

\item Every geodesic between $x$ and $y$ is a $D$--weakly tight geodesic for some $D$. This is because of Lemma 5.3 in the paper of Bestvina--Bromberg--Fujiwara \cite{BBF}, which states that given $a,b\in C(S)$ there are only finitely many subsurfaces $Z$ with $d_{Z}(a,b)>3$, i.e., there are only finitely many subsurfaces to check.

\item Every $D$--weakly tight geodesic is a $D'$--weakly tight geodesic whenever $D\leq D'$.
\end{itemize}

\end{remark}

By Lemma \ref{ti}, we have 

\begin{corollary}\label{2015}
Every tight geodesic is an $M$--weakly tight geodesic. 
\end{corollary}

We define similar notations for the slices on $D$-weakly tight geodesics.
\begin{definition}
Suppose $a,b\in C(S)$, $A,B\subseteq C(S)$, and $r>0$. 
\begin{itemize}
\item Let $\L_{WT}^{D}(a,b)$ be the set of all $D$--weakly tight geodesics between $a$ and $b$. 

\item Let $\d G^{D}(a,b)=\{v\in C(S)| v\in g\in \L_{WT}^{D}(a,b)\}.$ 

\item Let $\d G^{D}(A,B)=\bigcup_{a\in A, b\in B} G^{D}(a,b)$ and $G^{D}(a,b;r)=G^{D}(N_{r}(a), N_{r}(b))$. 
\end{itemize}
\end{definition}

We observe

\begin{theorem}\label{weaktight}
\it{Suppose $\x(S)\geq1$. Let $a,b\in C(S)$, $r\geq 0$, and $D\geq M$.}
\begin{enumerate}
\item \it{If $c\in g_{a,b}$, then $ |G^{D}(a,b)\cap N_{\delta}(c)| \leq N_{S}(2D,3).$}

\item \it{Suppose $d_{S}(a, b)\geq2r+2j+1$, where $j =3\delta+2$. If $c\in g_{a,b}$ and $c \notin N_{r+j}(a)\cup N_{r+j}(b)$, then $|G^{D}(a,b;r)\cap N_{2\delta}(c)|\leq N_{S}(2(D+M),3).$}
\end{enumerate}
\end{theorem}
\begin{proof}
By the same arguments in Corollary \ref{c} and Lemma \ref{cno3}, we observe that $G(a,b)\cap N_{\delta}(c)$ satisfies $\P(2D,3)$ for the first statement, and observe that $G(a,b;r)\cap N_{2\delta}(c)$ satisfies $\P(2(D+M),3)$ for the second statement. Now, we apply Theorem \ref{A} with $k=3$, and we are done.
\end{proof}
\begin{remark}
Even though Theorem \ref{B} states about the cardinalities of the slices on tight geodesics, we indeed counted the cardinalities of the slices on $M$--weakly tight geodesics. 
\end{remark}

We end this paper with the following questions.

Webb's bounds in Theorem \ref{webb's} is sharp; he gave lower bounds by giving examples. 
We ask the following question, which could fill up some gap between Webb's bounds in Theorem \ref{webb's} and our bounds in Theorem \ref{B}. 
\begin{question}\label{QQ2}
Are there sufficiently many $M$--weakly tight geodesics which are not tight geodesics?
\end{question}

Before we ask the last question, we observe Lemma \ref{Miyake1} and Lemma \ref{Miyake2}.
Suppose $a,b \in C(S)$; we let $\L(a,b)$ be the set of all geodesics between $a$ and $b$. 


\begin{lemma}\label{Miyake1}
Suppose $\x(S)=1.$ Let $a,b \in C(S)$. We have
$$ \L_{T}(a,b) = \L_{WT}^{M}(a,b) = \L(a,b) .$$
\end{lemma}
\begin{proof}
It follows from Definition \ref{e} and Corollary \ref{2015}.
\end{proof}

\begin{lemma}\label{Miyake2}
Suppose $\x(S)>1$ and $D'>M$. Let $a,b \in C(S)$. We have
$$ \L_{T}(a,b) \subseteq    \L_{WT}^{M}(a,b) \subseteq   \L_{WT}^{D'}(a,b) \subseteq  \L(a,b)  = \lim_{D\rightarrow \infty} \L_{WT}^{D}(a,b) .$$
\end{lemma}
\begin{proof}
It follows from Remark \ref{2016} and Corollary \ref{2015}.
\end{proof}

We ask more general questions.
\begin{question}\label{QQ3}
Suppose $\x(S)>1.$ Let $a,b \in C(S)$ and $D'>M$.
By Theorem \ref{MMMM}, we know that $\L_{T}(a,b)\neq \emptyset$. 
\begin{itemize}
\item $\{\L_{WT}^{D'}(a,b) \setminus \L_{T}(a,b) \}\neq \emptyset$? More specifically, what is the smallest $D'$ such that $\{\L_{WT}^{D'}(a,b) \setminus \L_{T}(a,b) \}\neq \emptyset$?

\item More generally, let $P,Q\in \mathbb{N}$ such that $P>Q$; $\{\L_{WT}^{P}(a,b) \setminus \L_{WT}^{Q}(a,b) \}\neq \emptyset$?

\item If the answers for the above questions are affirmative, is there a canonical way to construct $g\in \{\L_{WT}^{D'}(a,b) \setminus \L_{T}(a,b) \}$ and $g\in \{\L_{WT}^{P}(a,b) \setminus \L_{WT}^{Q}(a,b) \}$? (Masur--Minsky presented a canonical way to construct tight geodesics in \cite{MM2}.)

\end{itemize}
\end{question}

\bibliographystyle{abbrv}
\bibliography{references.bib}

\end{document}